\documentclass[12pt,leqno]{amsart}
\usepackage{epsfig}
\usepackage{amscd}

\usepackage{amsmath}
\usepackage{amsfonts,amssymb}
\usepackage{euscript}
\usepackage{textcomp}
\usepackage{amsthm}
\usepackage[matrix,arrow,curve]{xy}

\newtheorem{theo}{Theorem}[section]
\newtheorem{lemma}[theo]{Lemma}
\newtheorem{defi}[theo]{Definition}
\newtheorem{prop}[theo]{Proposition}

\newtheorem{cor}[theo]{Corollary}
\newtheorem{remark}[theo]{Remark}

\numberwithin{equation}{section}

\def\R{\mathbb{R}}

\def\Z{\mathbb{Z}}

\def\coh{\operatorname{coh}}

\def\O{{\mathcal O}}

\def\pre-tr{\operatorname{pre-tr}}

\def\Hom{\operatorname{Hom}}
\def\End{\operatorname{End}}

\newcommand{\bbR}{{\mathbb R}}

\newcommand{\bbZ}{{\mathbb Z}}

\newcommand{\bbP}{{\mathbb P}}

\newcommand{\cO}{{\mathcal O}}

\newcommand{\cL}{{\mathcal L}}

\newcommand{\cA}{{\mathcal A}}

\newcommand{\cE}{{\mathcal E}}

\newcommand{\cR}{{\mathcal R}}
\newcommand{\cS}{{\mathcal S}}

\newcommand{\Perf}{\operatorname{Perf}}

\newcommand{\im}{\operatorname{im}}

\newcommand{\Ext}{\operatorname{Ext}}

\newcommand{\coker}{\operatorname{coker}}

\renewcommand{\k}{\mathsf k}
\newcommand{\TT}{\mathcal T}
\newcommand{\mmod}{\mathrm{{-}mod}}
\newcommand{\modd}{\mathrm{mod{-}}}

\newcommand{\Modd}{\mathrm{Mod{-}}}
\renewcommand{\AA}{\mathcal A}

\newcommand{\s}{\sigma}

\newcommand{\xra}{\xrightarrow}
\renewcommand{\le}{\leqslant}
\renewcommand{\ge}{\geqslant}

\newcommand{\bul}{\bullet}

\DeclareMathOperator{\Supp}{\mathrm{Supp}}
\DeclareMathOperator{\qcoh}{\mathrm{qcoh}}

\begin{document}
%%%%%%%%%%%%%%%%%%%%%%%%%%

\title[Thick subcategories on curves]{Thick subcategories on curves}

\author{Alexey Elagin}
%\thanks{The first author was partially supported by HSE University Basic Research %Program, Russian Academic Excellence Project ``5--100'' and the Simons-IUM fellowship.}
\address{Institute for Information Transmission Problems (Kharkevich Institute), Moscow, RUSSIA;
National Research University Higher School of Economics,
RUSSIA
}
\email{alexelagin@rambler.ru}

\author{Valery A.~Lunts}
%\thanks{The second author was partially supported by Laboratory of Mirror Symmetry NRU %HSE, RF Government grant, ag.  14.641.31.0001.}
\address{Department of Mathematics, Indiana University,
Bloomington, IN 47405, USA;
National Research University Higher School of Economics,
RUSSIA
} 
\email{vlunts@indiana.edu}

\begin{abstract} We classify triangulated categories that are equivalent to finitely generated thick subcategories $\TT \subset D^b(\coh C)$ for smooth projective curves $C$ over an algebraically closed field.
\end{abstract}

\maketitle
\footnotetext[0]{The first author is partially supported within the framework of the HSE University Basic Research Program, Russian Academic Excellence Project '5-100' and the Simons-IUM fellowship.
The second author is partially supported by Laboratory of Mirror Symmetry NRU HSE, RF Government grant, ag. \textnumero 14.641.31.0001.}

\section{Introduction}

Let $C$ be a smooth projective curve of genus $g$ over an algebraically closed field~$\k$. We classify all finitely generated thick (triangulated) subcategories of $D^b(\coh C)$. Namely we prove that all such subcategories $\TT$ (if $\TT\neq 0,D^b(\coh C)$) are \emph{quiver-like}, that is there is a finite quiver $Q$ and an equivalence of categories
$$D_0^b(Q)\stackrel{\sim}{\to}\TT$$
where $D_0^b(Q)\subset D^b(Q)$ is the full triangulated subcategory generated by the simple modules corresponding to vertices (Theorem \ref{main theo}).

We then classify the quivers $Q$ which can be \emph{realized} on curves in this way (Theorem \ref{answer question one}). We also show that if $Q$ and $Q'$ are realizable quivers and there is an equivalence
$$D^b_0(Q)\simeq D^b_0(Q')$$
then $Q\simeq Q'$ (Corollary \ref{cor_unique}).

As a byproduct we obtain the following result (Propositions \ref{g=0 and g=1}, \ref{g=2}):

\textit{If $g\geq 2$, there is an infinite  descending binary tree of finitely generated thick subcategories of $D^b(\coh C)$. On the other hand, if $g=0,1$, no infinite descending chain of such subcategories exists (as follows easily from classical results of Grothendieck and Atiyah)}.

This phenomenon should be compared with the case of $0$-dimensional schemes. Namely, let $R$ be an artinian algebra. If $R$ is a complete intersection, then there are no infinite descending chains of finitely generated thick subcategories of $D^b(R \mmod)$, see~\cite{CI}. On the other hand in \cite{EL} simple examples of a non-complete intersection $R$ are constructed, such that there exists a descending binary tree of such subcategories.

The paper is organized as follows. Section 2 contains a brief reminder on triangulated categories and enhancements. In Section 3 we study quiver-like categories. In Section 4 we formulate our main observation: all thick subcategories on curves are quiver-like. In Section 5 we classify quivers which are realizable on curves (Definition \ref{def realizable}).

\section{A reminder about generation and enhancements of triangulated categories}

We fix a field $\k$. All our categories are $\k$-linear.
References for triangulated and dg categories include \cite{BK}, \cite{BoNe}, \cite{BLL}, \cite{Dr}, \cite{ELO}, \cite{Ke}.

If $\TT$ is a triangulated category and $X,Y$ are objects in $\TT$, we will freely use the equivalent notation for the corresponding space of morphisms
$$\Hom (X,Y[n])=\Hom ^n(X,Y)=\Ext ^n(X,Y).$$
If $X=Y$, then we also consider the graded algebra
$$\Ext ^\bullet _{\TT}(X,X)=\Ext ^\bullet (X,X)=\bigoplus _n\Ext ^n(X,X).$$

A triangulated category is \emph{$\Ext$-finite} if the space
$\bigoplus _n\Hom (X,Y[n])$
is finite dimensional for all objects $X,Y$.

We say that a triangulated category $\TT$ is \emph{non-split  generated} by a collection of objects $X_1,\ldots ,X_n$ if $\TT$ is the smallest full triangulated subcategory of $\TT$ which contains the objects $X_1,\ldots ,X_n$. We denote this $\TT=[X_1,\ldots ,X_n]$.

We say that a triangulated category $\TT$ is (split) \emph{generated} by a collection of objects $X_1,\ldots ,X_n$ if $\TT$ is the smallest full triangulated subcategory of $\TT$ which contains the objects $X_1,\ldots ,X_n$ and which is closed under direct summands in $\TT$. We denote this $\TT=\langle X_1,\ldots ,X_n\rangle$.

A \emph{thick} subcategory of a triangulated category $\TT$ is a full triangulated subcategory which is closed under direct summands in $\TT$.

A category is \emph{Karoubian} if it is idempotent complete, i.e. if every idempotent splits. Note that a thick subcategory of a Karoubian triangulated category is also Karoubian. If a triangulated category $\TT$ is closed under countable direct sums, then it is Karoubian.

For an abelian category $\AA$ we denote by $D(\AA)$ its (unbounded) derived category and by  $D^b(\AA)$ its bounded derived category, these categories are triangulated. If $\AA$ has countable direct sums then $D(\AA)$ is  Karoubian. Therefore all thick subcategories of $D(\AA)$ are also Karoubian.

If $\cS$ is a dg category we denote by $[\cS]$ its \emph{homotopy} category. If the dg category~$\cS$ is \emph{pre-triangulated}, then $[\cS]$ is triangulated.

An \emph{enhancement} of a triangulated category $\TT$ is a pre-triangulated dg category $\cS$ together with an equivalence of triangulated categories $[\cS]\stackrel{\sim}{\to}\TT$.
This allows us to consider an object $X \in \TT$ as an object in the dg category $\cS$. Then we denote its endomorphism dg algebra by
$$\bbR \End (X):=\End _{\cS}(X).$$
This dg algebra is well defined up to a quasi-equivalence and its cohomology algebra is
$$H^\bullet (\bbR \End (X))=\Ext ^\bullet _{\TT}(X,X).$$

For a dg algebra $\cE$ we consider the triangulated category $\Perf (\cE)$. This is the thick subcategory in $D(\cE)$ (= the derived category of right dg $\cE$-modules) which is (split) generated by the dg $\cE$-module $\cE$. The triangulated category $\Perf (\cE)$ is Karoubian and it has a natural enhancement. If dg algebras $\cE$ and $\cE'$ are quasi-isomorphic, then the triangulated categories $\Perf (\cE)$ and $\Perf (\cE ')$ are equivalent. A dg algebra is \emph{formal} if it is quasi-isomorphic to its cohomology graded algebra.

We will often use the following standard fact.

\begin{prop} \label{standard equiv} Let $\TT$ be a Karoubian triangulated category which has an enhancement. Assume that $\TT$ is generated by an object $X$, i.e. $\TT =\langle X\rangle$. Consider the dg algebra $\cE =\bbR \End (X)$. Then there exists a natural equivalence of categories
$$\TT\simeq \Perf (\cE).$$
\end{prop}

Finally, for a ring $R$ we denote by $\Modd R$ (resp. $\modd R$) the category of right $R$-modules (resp. finitely generated right $R$-modules).

\section{Quiver-like triangulated categories}
A \emph{quiver} means a \emph{finite} quiver, i.e. the set of vertices and arrows is finite.

Let $Q$ be a quiver with $n$ vertices $v_1,\ldots ,v_n$. Let $\k Q$ be the corresponding (hereditary) path algebra. Denote by $D^b(Q)=D^b(\Modd \k Q)$ the bounded derived category of right $\k Q$-modules. Since the algebra $\k Q$ is hereditary, every object in
$D^b(Q)$ is isomorphic to the direct sum of its cohomology.

%We denote by $D_0^b(Q)\subset D^b(Q)$ the full subcategory generated by $s_i$'s, i.e.
%$$D^b_0(Q)=\langle s_1,\ldots ,s_n\rangle.$$
Let $\cR \subset \k Q$ denote the radical of $\k Q$, i.e. $\cR $ is the $2$-sided ideal generated by all arrows. Denote by $(\modd \k Q)_{\cR }$ the abelian category  of finitely generated $\cR $-torsion $\k Q$-modules. (These modules are automatically finite dimensional.) 
Define $D^b_0(Q)\subset D^b(Q)$ as the subcategory of all finite complexes with cohomology in
$(\modd \k Q)_{\cR }$.
The triangulated categories $D^b(Q)$ and $D^b_0(Q)$ are Karoubian.

Let $s_1,\ldots ,s_n$ be the simple $\k Q$-modules, corresponding to the vertices.
For future reference we record the following easy fact.

\begin{lemma} \label{triv lemma} In the above notation the following holds:

(1) Every object in $(\modd \k Q)_{\cR }$ has a finite filtration with $s_i$'s as subquotients. The corresponding associated graded module is independent of the filtration and
\begin{equation}
\label{eq_dimvect}
K_0((\modd \k Q)_{\cR })\simeq\bigoplus _i\bbZ [s_i].
\end{equation}

(2) Every object in $D_0^b(Q)$ is isomorphic to the direct sum of its cohomology, and
$$K_0(D_0^b(Q))\simeq\bigoplus _i\bbZ [s_i].$$

(3) $D_0^b(Q)=\langle s_1,\ldots ,s_n\rangle=[s_1,\ldots ,s_n]$, i.e. $D^b_0(Q)$ is non-split generated by the $s_i$'s.
\end{lemma}

\begin{proof}
(1) Let $M\in (\modd \k Q)_{\cR }$, then the quotients of the (finite) filtration
$M\supset M\cdot \cR \supset M\cdot \cR ^2\supset\ldots$ are directs sums of $s_i$'s.
The isomorphism \eqref{eq_dimvect} is given by the dimension vector of $\k Q$-module.
(2) follows from $\mathrm{gldim}(\k Q)=1$.
In (3), inclusions $D_0^b(Q)\supset \langle s_1,\ldots ,s_n\rangle\supset [s_1,\ldots ,s_n]$ are obvious, while $D_0^b(Q)\subset [s_1,\ldots ,s_n]$ follows from (1) and (2).
%the Jordan-Holder theorem for the category $(\modd \k Q)_{\cR }$ and the fact that $gldim(\k Q)=1$. **check!!** **complete proof!!**
\end{proof}

We will not need the following lemma but include it here for the interested reader.

\begin{lemma}
The natural functor
$$\Psi\colon D^b((\modd \k Q)_{\cR })\to D^b_0(Q)$$
 is an equivalence.
\end{lemma}
\begin{proof}
Note that both categories are Karoubian and generated by the object $\oplus s_i$. Moreover, $\Psi(s_i)\simeq s_i$ for any $i$. Therefore to prove that $\Psi$ is an equivalence it suffices (using the standard devissage technique) to check that~$\Psi$ induces isomorphisms
$$\Ext^m_{(\modd \k Q)_{\cR }}(s_i,s_j)=\Hom^m_{D^b((\modd \k Q)_{\cR })}(s_i,s_j)\to
\Hom^m_{D^b_0(Q)}(s_i,s_j)=
\Ext^m_{\Modd \k Q}(s_i,s_j)$$
for any $i,j$ and $m\in\bbZ$.
For $m<0$ this is clear; for  $m=0$ this holds since $(\modd \k Q)_{\cR }\subset \Modd \k Q$ is a full subcategory. For $m=1$ this holds by Yoneda's description of $\Ext^1$ groups since the subcategory $(\modd \k Q)_{\cR }\subset \Modd \k Q$ is extension-closed.
For $m\ge 2$ we have $\Ext^m_{\Modd \k Q}(s_i,s_j)=0$ since $\k Q$ is hereditary, let us check that $\Hom^m_{D^b((\modd \k Q)_{\cR })}(s_i,s_j)=0$ for $m\ge 2$.

By definition, any morphism $f\colon s_i\to s_j[m]$ in $D^b((\modd \k Q)_{\cR })$ has the form
$$s_i\xleftarrow{q} C^\bul\xrightarrow{p} s_j[m],$$
where $C^\bul$ is a bounded complex over $(\modd \k Q)_{\cR }$, $p,q$ are homomorphisms of complexes and $q$ is a quasi-isomorphism. We claim that there exists a complex $P^\bul=[P^{-1}\to P^0]$ over $(\modd \k Q)_{\cR }$ and a quasi-isomorphism $s\colon P^\bul\to C^\bul$. Then
$$f=pq^{-1}=pss^{-1}q^{-1}=0$$
since $ps=0$ (recall that $m\ge 2$).

To prove the claim, let $\bar P^\bul=[\bar P^{-1}\xra{d} \bar P^0]$ be a resolution of $s_i$ by projective finitely generated $\k Q$-modules. There exists a quasi-isomorphism $\bar s\colon \bar P^\bul\to C^\bul$.
Since $C^k$ are $\cR$-torsion modules, one can take $N$ such that $C^k\cdot \cR^{N-1}=0$ for $k=0,-1$. We let now $P^0:=\bar P^0/(\bar P^0\cdot \cR^N)$. By assumptions, $\bar s^0\colon \bar P^0\to C^0$ factors via $P^0$. Recall that $H^\bul(\bar P)\simeq s_i$, hence $d$ is injective (we will treat $\bar P^{-1}$ as a submodule in $\bar P^0$) and $$\bar P^0\cdot \cR^{N}\subset \bar P^0\cdot \cR\subset \bar P^{-1}.$$
Let $P^{-1}:=\bar P^{-1}/(\bar P^0\cdot \cR^N)$. Then $P^\bul$ is quasi-isomorphic to $\bar P^\bul$ (and to $s_i$).  Clearly $P^\bul$ is a complex over $(\modd \k Q)_\cR$. Also
$$\bar s^{-1}(\bar P^0\cdot \cR^N)\subset \bar s^{-1}(\bar P^{-1}\cdot \cR^{N-1})\subset \bar s^{-1}(\bar P^{-1})\cdot \cR^{N-1}\subset C^{-1}\cdot \cR^{N-1}=0,$$
hence $\bar s^{-1}\colon \bar P^{-1}\to C^{-1}$ factors via $P^{-1}$. Therefore, $\bar s$ factors via a quasi-isomorphism $s\colon P^\bul\to C^\bul$.
This concludes the proof of the claim and the lemma.
\end{proof}

\begin{defi} \label{def quivet like} A triangulated category $\TT$ is called {\rm quiver-like} if there exists a finite quiver~$Q$ and an equivalence of triangulated categories
$D^b_0(Q)\simeq \TT$.
%Then we call the objects $\{\Phi (s_i)\}\subset \TT$ the {\rm vertex-like generators} of $\TT$.
\end{defi}

We obtain the immediate consequence of Definition \ref{def quivet like} and Lemma \ref{triv lemma}.

\begin{cor} \label{triv cor} Let $\TT$ be a quiver-like triangulated category with
an equivalence $\Phi \colon D^b_0(Q)\to \TT$. Put $t_i=\Phi (s_i)$. Then we have the following.

(1) For any indecomposable object $B\in \TT$ there exists a sequence of objects $B_0,\ldots ,B_m$ such that $B_m=B[d]$ for some $d\in \bbZ$,
$B_0=0$, and for each $i=1,\ldots ,m$ the object $B_i$ fits into an exact triangle
$$B_{i-1}\to B_i\to t_{j_i}\to B_{i-1}[1]$$
for some $j_i$.

(2) $K_0(\TT)\simeq\bigoplus _j\bbZ [t_j]$.

(3) $\TT=[t_1,\ldots ,t_n]$, i.e. $\TT$ is non-split generated by the $t_j$'s.
\end{cor}

\begin{lemma} \label{lemma on formality}
Let $\cA$ be an abelian category in which for every object $A$ there exists an injective resolution
$$0\to A\to I^0\to I^1\to 0.$$
Let $A_1,\ldots,A_n$ be objects in $\cA$ such that  $\Hom (A_i,A_j)=\delta_{ij}\cdot \k$   and $\Ext ^s(A_i,A_j)=0$ for all $i,j$ and $s\ne 0,1$. Then the dg algebra $\bbR\End (\oplus_{i=1}^n A_i)$ is formal.

The same holds for projective resolutions.
\end{lemma}
\begin{proof}
Choose injective resolutions $A_j\to I_j^\bullet$ of length $1$. Then the dg algebra $\bbR \End (\oplus_j A_j)$ is quasi-isomorphic to
$$\cE  :=\End (\oplus_j I_j^\bullet)=\cE  ^{-1}\oplus \cE  ^0\oplus \cE  ^1.$$
Let $e_j\in \cE  ^0$ be the idempotent of the summand $I_j^\bullet$.
%Choose a subspace $V=V^0\oplus V^1$ of $\cE $, where $V^0=\oplus_j \k e_j\subset \cE  %^0$ and $V^1\subset \cE  ^1$ is any complement of $d(\cE  ^0)\subset \cE  ^1$.
Then $\cE  ^0=\oplus_{i,j}e_i\cE  ^0 e_j$, $\cE  ^1=\oplus_{i,j}e_i\cE_1e_j$
and $d$ sends $e_i\cE_0e_j$ to $e_i\cE_1e_j$.
Choose a subspace $V=V^0\oplus V^1$ of $\cE $, where $V^0=\oplus_j \k e_j\subset \cE  ^0$,   $V^1_{i,j}\subset e_i\cE^1e_j$ is any complement of $d(e_i\cE  ^0e_j)\subset e_i\cE  ^1e_j$,
and $V^1=\oplus_{i,j} V^1_{i,j}\subset \cE  ^1$.
Then $V$ is a dg subalgebra of $\cE  $ and the inclusion $V\subset \cE $ is a quasi-isomorphism.

In case of projective resolutions the proof is similar.
\end{proof}

\begin{cor} \label{cor formal}Let $Q$ be a quiver and consider the category $D^b_0(Q)$ with the standard enhancement coming from the embedding $D^b_0(Q)\subset D^b(Q)$. Then the dg algebra
$$\bbR \End (\oplus s_i)$$
is formal. Hence there is an equivalence of triangulated categories
$$D^b_0(Q)\simeq \Perf (\Ext ^\bullet (\oplus s_i,\oplus s_i))$$
where the graded algebra $\Ext ^\bullet (\oplus s_i,\oplus s_i)$ is considered as a dg algebra with zero differential.
\end{cor}

\begin{proof} The formality of the dg algebra $\bbR \End (\oplus s_i)$ follows from Lemma \ref{lemma on formality} applied to the abelian category $\cA$ of all right $\k Q $-modules and taking $A_i=s_i$.

Because the category $D^b_0(Q)$ is Karoubian by Proposition \ref{standard equiv} we get the equivalence
$$D^b_0(Q)\simeq \Perf (\bbR \End (\oplus s_i)).$$
The last assertion then follows from the fact that the dg algebras $\Ext ^\bullet (\oplus s_i,\oplus s_i)$ and
$\bbR \End (\oplus s_i)$ are quasi-isomorphic.
\end{proof}

\begin{defi}\label{defi vertex like} Let $\TT$ be a triangulated category with an enhancement. A collection of objects $\{t_1,\ldots ,t_n\}$ is $\TT$ is called {\rm vertex-like} if the endomorphism dg algebra
$$\bbR \End (\oplus t_i)$$
is formal, and in addition $\Hom (t_i,t_j)=\delta _{ij}\cdot \k$, and $\Hom ^p(t_i,t_j)=0$ for all $i,j$ and $p\neq 0,1$.
\end{defi}

\begin{remark} It follows from Corollary \ref{cor formal} that the collection of objects $\{ s_1,\ldots ,s_n\}$ in the category $D^b_0(Q)$ is vertex-like. Note that the dimension of the space $\Ext ^1(s_i,s_j)$ is equal to the number of arrows from $v_j$ to $v_i$.
\end{remark}

The next proposition gives a necessary and sufficient condition for a category to be quiver-like.

\begin{prop} \label{equiv cond} Let $\TT$ be an $\Ext$-finite triangulated category. The following conditions are equivalent.

(1) $\TT$ is quiver-like.

(2) $\TT\simeq \Perf(E)$ where $E=E^0\oplus E^1$ is a dg algebra with zero differential and such that $E^0=\k \times \ldots \times \k$.

(3) $\TT$ is Karoubian, it has an enhancement and it is generated by a collection of objects $\{t_1,\ldots ,t_n\}$ that is vertex-like.

Moreover, if $\TT$ satisfies (3), then there exists a quiver $Q$ and an equivalence $\Phi \colon D^b_0(Q)\to \TT$ such that $\Phi (s_i)=t_i$.
\end{prop}

\begin{proof} (1)$\Rightarrow$(2) is contained in Corollary \ref{cor formal}.

(2)$\Rightarrow$(3). Since $\TT\simeq \Perf(E)$, it is Karoubian and has an enhancement. Let $e_1,\ldots ,e_n\in E$ be the idempotents corresponding to the factors in $E^0=\k \times \ldots \times \k$. Then the right dg $E$-modules $e_iE$ are h-projective, they generate $\Perf (E)$, and the dg algebra
$$\bbR \End (\bigoplus _ie_iE)=\bbR \End (E)=E$$
is formal.
In addition $\Hom (e_iE,e_jE)=\delta _{ij}\cdot \k$, $\Hom ^p(e_iE,e_jE)=0$ for all $i,j$ and $p\neq 0,1$. So we can take $t_i=e_iE$.

(3)$\Rightarrow$(1). Consider the graded algebra $\cE =\Ext ^\bullet (\oplus t_i,\oplus t_i)$ as a dg algebra with zero differential. Our assumptions imply that the category $\TT$ is equivalent to the category $\Perf (\cE)$.

Now define the quiver $Q$ with vertices $v_1,\ldots ,v_n$
and the number of arrows from $v_i$ to $v_j$ equal to $\dim \Hom (t_j,t_i[1])$. Let $s_1,\ldots ,s_n\in D_0^b(Q)$ be the corresponding simple modules.
By construction, we have $\Ext^\bullet(t_i,t_j)\simeq\Ext^\bullet(s_i,s_j)$ for all $i,j$.
By Corollary \ref{cor formal} we get
$$D_0^b(Q)\simeq \Perf(\Ext^\bullet(\oplus s_i,\oplus s_i))\simeq \Perf (\cE)\simeq \TT,$$
i.e. $\TT$ is quiver-like. This proves the implication (3)$\Rightarrow$(1) and also the last assertion of the proposition.
\end{proof}

\begin{prop}
\label{prop_niceexist}
Assume that the field $\k$ is algebraically closed. Let $\AA$ be an abelian ($\k$-linear) category such that any object  $A\in \AA$ has an injective resolution of length $\le 1$. Assume that the category $D^b(\AA)$ is Karoubian. Let $\TT\subset D^b(\AA)$ be a finitely generated $\Ext$-finite thick subcategory. Assume there exists a linear function
$$r\colon K_0(\TT)\to \Z,$$
such that for any nonzero $F\in \TT\cap\AA$ one has $r([F])>0$. Then the category $\TT$ satisfies the condition (3) in Proposition \ref{equiv cond}, and hence it is quiver-like.

The same holds for  projective resolutions.
\end{prop}
\begin{proof}
Note that $\AA$ is hereditary and thus any object in $\TT$ is a direct sum of its cohomology. It follows that one can choose a finite set of generators in $\TT$ belonging to $\AA\subset D^b(\AA)$.
Take any family $A_1,\ldots,A_n\in \AA$ of nonzero objects generating $\TT$
such that
%for any other such family $B_1,\ldots, B_m$ one has
\begin{enumerate}
%\label{eq_minimal}
\item $\sum_i r([A_i])$ is the minimal possible;
\item the number $n$ is the maximal possible among all families with the fixed $\sum_i r([A_i])$.
\end{enumerate}
(such a family exists because $r([A])>0$ for any nonzero $A\in\TT\cap \AA$). Note that $A_i\ncong A_j$ for $i\neq j$.
We claim that the family $\{ A_1,\ldots,A_n\}$ is vertex-like.

First we check that there are no morphisms between $A_i$'s except for scalar multiplication. Let
$f\colon A_i\to A_j$ be a morphism. Denote $K:=\ker f,I:=\im f, C:=\coker f$. Since $\AA$ is hereditary, the complex $Cone (f)$ in $\TT$ is quasi-isomorphic to $K[1]\oplus C$. Since $\TT$ is thick we get that $K,C\in \TT$. Also we get $I\in \TT$ and
$$\langle A_i,A_j\rangle=\langle K,I,C\rangle.$$
If $i\ne j$ we have
$$r([A_i])+r([A_j])=r([K])+r([I])+r([I])+r([C])=(r([K])+r([I])+r([C]))+r([I]).$$
Replacing $A_i,A_j$ with $K,I,C$ we get a generating family with the smaller $\sum_i r([A_i])$ unless $I=0$, it contradicts to condition (1). Hence $f=0$.
If $i=j$ we get
$$r([A_i])=r([K])+r([I])=r([I])+r([C]).$$
Replacing $A_i$ with $K,I$ we get a generating family with the same $\sum_i r([A_i])$ and with the bigger number of objects unless $I=0$ or $K=0$, it contradicts to condition (2). Hence $f=0$ or $K=0$. Similarly, replacing $A_i$ with $I,C$ we see that $I=0$ or $C=0$.
Thus, if $f\ne 0$ then $K=C=0$ and $f$ is an isomorphism. We proved that for each $i$ the endomorphism algebra $\End A_i$ is a finite-dimensional division $\k$-algebra. Since $\k$ is algebraically closed, $\End A_i=\k$.

Clearly, $\Hom^s(A_i,A_j)=0$ for $s\ne 0,1$. Finally, the dg algebra $\R\End(\oplus_i A_i)$ is formal by Lemma~\ref{lemma on formality}. Therefore $\{ A_1,\ldots,A_n\}$ is a vertex-like collection. Also $\TT$ is Karoubian, hence it satisfies the condition (3) of Proposition \ref{equiv cond}.
\end{proof}

%\begin{cor}
%Let $\TT$ be a quiver-like triangulated category. Then any finitely generated thick subcategory $\TT '\subset \TT$ is also quiver-like.
%\end{cor}

%\begin{proof} We may assume that $\TT=D^b_0(Q)$ for a quiver $Q$.
%Then in Proposition \ref{prop_niceexist} take $\AA$ to be (for example) the abelian category of finite dimensional $\k[Q]$-modules.
%The category $D^b(\AA)$ is Karoubian.
%\end{proof}

\begin{cor}\label{first cor}
Let $A$ be a hereditary $\k$-algebra over an algebraically closed field. Let $\TT=\langle M_1,\ldots,M_n\rangle\subset D^b(\Modd A)$ be any thick subcategory generated by finite-dimensional (over $\k$) modules. Then $\TT$ is quiver-like.
\end{cor}
\begin{proof}
We use Proposition~\ref{prop_niceexist}. Namely, we take $\AA=\Modd A$ and $r$ to be the function induced by the dimension of a module over $\k$.
\end{proof}

\begin{cor}
Let $\TT$ be a quiver-like triangulated category. Then any finitely generated thick subcategory $\TT '\subset \TT$ is also quiver-like.
\end{cor}

\begin{proof} We may assume that $\TT=D^b_0(Q)$ for a quiver $Q$ and use Corollary \ref{first cor} with $A=\k Q$.
\end{proof}

\medskip

\begin{prop}
\label{prop_Qtree}
Let $Q$ be a quiver with two vertices $1,2$ such that for any $i, j\in\{1,2\}$ there is at least one arrow from $i$ to $j$. Then the category $D^b_0(Q)$ has an infinite descending binary tree of thick subcategories. Moreover one can find such a tree with the following additional property: if $\TT _1$ and $\TT _2$ are two elements of this tree which are not located one above the other, then
$\TT _1\cap \TT _2=0$.
\end{prop}
\begin{proof}
Let $a\colon 1\to 1, b\colon 2\to 1, c\colon 2\to 2$ be some arrows.
Define the right $\k Q$-module $M^{(1)}$ as follows: $M^{(1)}_1=\k x,$  $M^{(1)}_2=\k y$ with $x\cdot b=y$ and all other arrows in $Q$ acting by zero. Define another right $\k Q$-module $M^{(2)}$ as $M^{(2)}_1=\k \alpha \oplus \k\beta,$  $M^{(2)}_2=\k\gamma \oplus \k\delta$ with the nontrivial action of the arrows given by $\alpha \cdot a=\beta$, $\beta \cdot b=\gamma$, $\gamma \cdot c=\delta$. Then one checks that
$$\Hom (M^{(i)},M^{(j)})=\delta _{ij}\cdot\k\quad \text{and}\quad \Ext^1 (M^{(i)},M^{(j)})\ne 0 \quad \text{for all}\quad  i,j\in\{1,2\}.$$
We conclude (using Lemma~\ref{lemma on formality}) that the modules $M^{(1)},M^{(2)}$ form a vertex-like set. It follows then from Proposition~\ref{equiv cond} and  Corollary \ref{triv cor} that $\langle M^{(1)},M^{(2)}\rangle =[M^{(1)},M^{(2)}]$. Consequently, the thick subcategory $\langle M^{(1)},M^{(2)}\rangle$ is strictly smaller than $D^b_0(Q)$ (because, for example, all objects in $\langle M^{(1)},M^{(2)}\rangle$ have even-dimensional cohomology).
Moreover, by Proposition~\ref{equiv cond} we get an equivalence $\langle M^{(1)},M^{(2)}\rangle\simeq D^b_0(Q')$ where the quiver $Q'$ also satisfies the assumptions of the present proposition. We  can iterate the process, which then gives an infinite descending chain of thick subcategories of $D^b_0(Q)$. To construct a required descending binary tree of subcategories we can proceed as follows.

For convenience let us describe the modules $M^{(1)},M^{(2)}$ constructed above by the diagrams
$$M^{(1)}:\ \bullet \stackrel{b}{\to}\bullet$$
$$M^{(2)}:\ \bullet \stackrel{a}{\to}\bullet \stackrel{b}{\to}\bullet \stackrel{c}{\to}\bullet$$
Let us similarly define the right $A$ modules
$$M^{(3)}:\ \bullet \stackrel{a}{\to} \bullet \stackrel{a}{\to}\bullet \stackrel{b}{\to}\bullet \stackrel{c}{\to}\bullet \stackrel{c}{\to}\bullet$$
$$M^{(4)}:\ \bullet \stackrel{a}{\to}\bullet \stackrel{a}{\to} \bullet \stackrel{a}{\to}\bullet \stackrel{b}{\to}\bullet \stackrel{c}{\to}\bullet \stackrel{c}{\to}\bullet \stackrel{c}{\to}\bullet$$
One checks that
\begin{equation}
\label{eq_he}
\Hom (M^{(i)},M^{(j)})=\delta _{ij}\cdot\k\quad\text{and}\quad
\Ext ^1(M^{(i)},M^{(j)})\ne 0
\end{equation}
for all $i,j\in \{1,2,3,4\}$. Indeed, for $\Hom$ this can be  done by hands, for $\Ext^1$ use
$$\chi(M^{(i)},M^{(j)})=i\cdot j\cdot\chi(S_1\oplus S_2, S_1\oplus S_2)=ij(2-|\{\text{arrows in $Q$}\}|)<0.$$
 Hence the thick subcategory $\langle M^{(3)},M^{(4)}\rangle \subset D^b_0(Q)$ is also quiver-like.

We claim that the categories
$\langle M^{(1)},M^{(2)}\rangle$ and $\langle M^{(3)},M^{(4)}\rangle$ have zero intersection. Assume the converse. Since any object in $D^b_0(Q)$ is a direct sum of its cohomology, it follows that there exists a nonzero indecomposable $\k Q$-module
$L\in \langle M^{(1)},M^{(2)}\rangle\cap \langle M^{(3)},M^{(4)}\rangle$.
By Corollary~\ref{triv cor}, every indecomposable $\k Q$-module in $\langle M^{(1)},M^{(2)}\rangle$ (resp. in $\langle M^{(3)},M^{(4)}\rangle$)  has a filtration with  subquotients $M^{(1)},M^{(2)}$ (resp. $M^{(3)},M^{(4)}$). Therefore, if $M$ (resp. $N$) is an indecomposable $\k Q$-module  in $\langle M^{(1)},M^{(2)}\rangle$ (resp. in $\langle M^{(3)},M^{(4)}\rangle$), then $\Hom (M,N)=0$ by \eqref{eq_he}.  In particular, $\Hom(L,L)=0$ and thus $L=0$, a contradiction.

It is clear that we can now iterate the process to construct a descending binary tree of quiver-like categories with the required properties.
\end{proof}

\section{Thick subcategories on curves}

In this section we assume that $C$ is a smooth projective connected curve over an algebraically closed field $\k$. Our goal is to classify thick subcategories in $D^b(\coh C)$.

\begin{lemma}
\label{lemma_notorsion}
Let $\TT\subset D^b(\coh C)$ be a thick subcategory which contains a nonzero vector bundle and a nonzero torsion sheaf. Then $\TT=D^b(\coh C)$.
\end{lemma}
\begin{proof} Let $V,T\in \TT$ be a vector bundle and a torsion sheaf respectively. Let $x\in \Supp (T)$. It is easy to see that the skyscraper sheaf $\O_x$ is in $\TT$.

Choose a line bundle $L$ and a surjection $V\to L$. This gives a short exact sequence of vector bundles
$$0\to E\to V\to L\to 0.$$
Choose a surjection $L\to \O_x$ and denote by $V^{(1)}\subset V$ the kernel of the composition $V\to L\to \O_x$. Thus $V^{(1)}\in \TT$ and we obtain a short exact sequence
$$0\to E\to V^{(1)}\to L(-x)\to 0.$$
Iterating this process we get for any $n\ge 1$ a short exact sequence
$$0\to E\to V^{(n)}\to L(-nx)\to 0$$
with $V^{(n)}\in \TT$. For $n>>0$ this sequence splits, hence $L(nx)\in \TT$ for some $n$. It is then easy so see that $L(nx)\in \TT$ for all $n\in \Z$.

Let $F\in \coh C$. We can find an exact sequence of coherent sheaves
$$0\to K\to \oplus L(nx)\to \oplus L(mx)\to F\to 0.$$
The two middle terms and in $\TT$ and the category $\coh C$ is hereditary. Hence  $F\in \TT$ as a direct summand of $Cone(\oplus L(nx)\to \oplus L(mx))$. Therefore $\coh C\subset \TT$. It follows that $\TT=D^b(\coh C)$.
\end{proof}

We obtain the immediate corollary.

\begin{cor}\label{def cor} If $\TT \subset D^b(\coh C)$ is a thick subcategory, $\TT\ne 0, D^b(\coh C)$, then exactly one of the following holds:

(1) Every object in $\TT$ has torsion cohomology.

(2) Every object in $\TT$ has torsion free cohomology.
\end{cor}

\begin{proof} Indeed, every object in $D^b(\coh C)$ is the direct sum of its cohomology. So it remains to apply Lemma \ref{lemma_notorsion}.
\end{proof}

\begin{defi} \label{tor nontor} We will say that a thick subcategory $\TT \subset D^b(\coh C)$
is \emph{proper} if $\TT\ne 0, D^b(\coh C)$. We call $\TT$ \emph{torsion} (resp. \emph{torsion-free}) in case (1) (resp. (2)) in Corollary \ref{def cor} holds.
\end{defi}

Now we can formulate our main observation.

\begin{theo} \label{main theo} Every finitely generated thick proper subcategory $\TT \subset D^b(\coh C)$  is quiver-like.
\end{theo}

\begin{proof} We consider the two cases of Definition \ref{tor nontor}.

Case 1: $\TT$ is torsion. In this case we may apply Proposition \ref{prop_niceexist} with $\AA =\qcoh C$ and the function
$r\colon K_0(\TT)\to \bbZ$ induced by the dimension (over $\k$) of a torsion sheaf.

Case 2: $\TT$ is torsion-free. In this case we again apply Proposition \ref{prop_niceexist} with $\AA =\qcoh C$,
but take the function
$r\colon K_0(\TT)\to \bbZ$ to be induced by the rank of a vector bundle.
\end{proof}

\begin{defi} \label{def realizable} A quiver $Q$ is called \emph{realizable} if
the category $D^b_0(Q)$ is equivalent to a thick finitely generated subcategory of $D^b(\coh C)$ for a smooth projective curve over an algebraically closed field $\k$.
\end{defi}

In the next section we are going to classify the realizable quivers.  For now let us give some examples.

\begin{defi}
\label{def_T0}
We denote by ${\bf Q}_m$ the
quiver with one vertex and $m$ loops.
\end{defi}

\begin{lemma} \label{first example for curves} For any $n$, the quiver which is the disjoint union of $n$ copies of the quiver ${\bf Q}_1$ is realizable on any curve. In fact any torsion category (Definition \ref{tor nontor}) supported at $n$ distinct points is equivalent to
$$D^b_0({\bf Q}_1\sqcup \ldots \sqcup {\bf Q}_1)=D^b_0({\bf Q}_1)^{\oplus n}.$$
\end{lemma}

\begin{proof} Let $p$ be a point on a smooth curve $C$. Then the sky-scraper sheaf $\cO _p\in D^b(\coh C)$ is a vertex-like object with
$$\Hom (\cO _p,\cO _p)\simeq \Ext ^1(\cO _p,\cO _p)\simeq \k.$$
Hence by Proposition \ref{equiv cond} the thick subcategory $\langle \cO _p\rangle \subset D^b(\coh C)$ is equivalent to $D^b_0({\bf Q}_1)$.  Now it is clear that the thick subcategory
$$\TT =\langle \cO _{p_1},\ldots ,\cO _{p_n}\rangle $$
for $n$ different points $p_1,\ldots ,p_n\in C$ is equivalent to
$$D^b_0({\bf Q}_1\sqcup \ldots \sqcup {\bf Q}_1)=D^b_0({\bf Q}_1)^{\oplus n}.$$
It remains to note that $\langle \cO_p\rangle$ contains no proper thick subcategories and thus any thick finitely generated torsion subcategory in $D^b(\coh C)$ is of the form $\langle \cO _{p_1},\ldots ,\cO _{p_n}\rangle$ for some points $p_1,\ldots ,p_n$.
\end{proof}

\begin{prop} \label{g=0 and g=1} Let $C$ be a curve of genus $g$ and let $\TT \subset D^b(\coh C)$ be a thick finitely generated proper subcategory.

(1) If $g=0$, then
$\TT \simeq D^b_0({\bf Q}_0)$ or $\TT\simeq D^b_0({\bf Q}_1\sqcup \ldots \sqcup {\bf Q}_1).$ The first case occurs if $\TT$ is torsion-free and the second one occurs when $\TT$ is torsion.

(2) If $g=1$, then $\TT\simeq D^b_0({\bf Q}_1\sqcup \ldots \sqcup {\bf Q}_1)$.

(3) If $g=0$ or $g=1$, the category $\TT$ contains only finitely many distinct thick subcategories.
\end{prop}

\begin{proof} (1) Let $g=0$. Assume that $\TT$ is torsion-free. Since every vector bundle on~$\bbP ^1$ is a direct sum of line bundles $\cO (n)$, it is easy to see that $\TT =\langle \cO (n)\rangle $ for some~$n$ and hence
$$\TT \simeq D^b(\modd \k) \simeq D^b_0({\bf Q}_0).$$
If on the other hand $\TT$ is torsion, it follows from Lemma \ref{first example for curves} that
$$\TT\simeq D^b_0({\bf Q}_1\sqcup \ldots \sqcup {\bf Q}_1).$$

(2) Let $g=1$ and let $F\in \TT$ be a nonzero indecomposable object.  Then there exists an autoequivalence $\Psi $ of $D^b(\coh C)$ such that $\Psi (F)$ is a torsion sheaf.

It follows that the category $\Psi (\TT)$ is torsion. Then again by Lemma~\ref{first example for curves} we find that $\TT\simeq \Psi (\TT)\simeq D^b_0({\bf Q}_1\sqcup \ldots \sqcup {\bf Q}_1)$.

(3) This follows from (1) and (2) and the fact that the categories $D^b_0({\bf Q}_0)$ and $D^b({\bf Q}_1)$ have no proper thick subcategories.
\end{proof}

It contrast to Proposition \ref{g=0 and g=1},  thick subcategories of curves of genus $g\geq 2$ behave differently.

\begin{prop} \label{g=2} (1) Let $C$ be a curve of genus $g\geq 2$. Then there exists an infinite descending binary tree of thick subcategories in $D^b(\coh C)$ with the following property: if elements $\TT _1$, $\TT _2$ of this tree are not located one above the other, then
$$\TT _1\cap \TT _2=0.$$

(2) For any $n\geq 0$ the quiver ${\bf Q}_n$ is realizable on a curve of genus $g=n$.
\end{prop}

\begin{proof} (1) Let $\cL_1,\cL_2$ be distinct line bundles of degree $0$ on $C$. Then for all $i,j\in \{1,2\}$ we have $\Hom (\cL_i,\cL_j)=\delta _{ij}\cdot \k$ and $\Ext^s(\cL_i,\cL_j)=0$ for $s\ne 0,1$. By Lemma~\ref{lemma on formality}, Definition \ref{defi vertex like} and  Proposition \ref{equiv cond} we know that $\{\cL_1,\cL_2\}$ is a vertex-like collection, and the thick subcategory  $\langle\cL_1,\cL_2\rangle\subset D^b(\coh C)$ is equivalent to $D^b_0(Q)$, where $Q$ is the quiver with vertices $v_1,v_2$ and $\dim\Ext^1(\cL_j,\cL_i)$ arrows from $v_i$ to $v_j$ for any $i,j\in\{1,2\}$. Note that for any $i,j\in\{1,2\}$ by the Riemann-Roch formula
$$\dim \Ext ^1(\cL_i,\cL_j)=g(C)-1+\delta_{ij}> 0,$$
hence by Proposition~\ref{prop_Qtree} the category $D^b_0(Q)$ has an infinite descending binary tree of thick finitely generated subcategories with the required property.

(2) It suffices to take any line bundle $\cL$ on $C$. The thick subcategory $\langle \cL \rangle \subset D^b(\coh C)$ is equivalent to $D_0^b({\bf Q}_g)$ where $g$ is the genus of $C$.
\end{proof}

\section{Realization of quivers on curves and uniqueness problems}

In this section a \emph{curve} means a smooth projective connected curve over an algebraically closed field $\k$. We complete the problem of classification of proper finitely generated thick subcategories of curves which we started in the previous section.
In view of Theorem \ref{main theo}, Lemma \ref{first example for curves}, and Proposition \ref{g=0 and g=1} it remains to answer the following questions:

\medskip

\noindent{\bf Q1}. Which quivers $Q$ are realizable by torsion-free subcategories on curves of genus $g\geq 2$? (Definitions \ref{def realizable}, \ref{tor nontor}).

\medskip

We can also ask the following question.

\medskip

\noindent{\bf Q2}. Suppose the quiver $Q$ is realizable. Is then $Q$ determined uniquely by the category
$D^b_0(Q)$?

\medskip

We start with question {\bf Q1}.

First let us summarize the relevant results from the previous sections.

\begin{prop} \label{useful summary for torsion-free}  Let $C$ be a curve and let $\TT \subset D^b(\coh C)$ be a proper thick subcategory which is torsion-free (Definition \ref{tor nontor}). Then

(1) $\TT$ is quiver-like.

(2) $\TT =\langle E_1,\ldots ,E_n\rangle$, where $\{E_1,\ldots ,E_n\}$ is a vertex-like collection of vector bundles on $C$.

(3) We have $\TT =[E_1,\ldots ,E_n]$ and $K_0(\TT)=\bigoplus _i\bbZ [E_i]$.

(4) Every indecomposable object in $\TT$ is of the form $F[m]$, where $F$ is a vector bundle that has a filtration with subquotients being $E_i$'s.

Vice versa, a vertex-like collection of vector bundles $\{E_1,\ldots ,E_n\}$ generates a torsion-free proper thick subcategory of $D^b(\coh C)$.
\end{prop}

\begin{proof} (1) Follows from Theorem \ref{main theo}. Then (2),(3),(4) follow from Proposition \ref{equiv cond} and Corollary  \ref{triv cor}.

For the last assertion: we know that $\langle E_1,\ldots ,E_n\rangle \subset D^b(\coh C)$ is a quiver-like subcategory of $D^b(\coh C)$. Now Corollary  \ref{triv cor} implies that all indecomposable objects of $\langle E_1,\ldots ,E_n\rangle$ are (shifted) vector bundles on $C$. Hence it is a proper torsion-free thick subcategory of $D^b(\coh C)$.
\end{proof}

Let $C$ be a curve of genus $g$. For a vector bundle $E$ on $C$ let $r(E)$ and $d(E)$ denote respectively its rank and degree. For vector bundles $E,F$
put
$$\chi (E,F)=\chi (C,E,F)=\dim \Hom (E,F)-\dim \Ext ^1(E,F).$$
By a version of Riemann-Roch formula we have
\begin{equation} \label{RR formula}
\chi (E,F)=r(E)r(F)(1-g)+r(E)d(F)-r(F)d(E).
\end{equation}

A finite quiver $Q$ with vertices $v_1,\ldots ,v_n$ is determined by a square matrix $A=(a_{ij})\in M_{n\times n}(\bbZ)$ with nonnegative entries $a_{ij}\geq 0$, such that $a_{ij}$ is the number of arrows from $v_i$ to $v_j$. Put $Q=Q(A)$. Recall (Proposition \ref{equiv cond}) that the quiver $Q(A)$ is realized by a torsion-free category on a curve $C$ if and only if there exists a vertex-like collection of vector bundles
$\{ E_1,\ldots ,E_n\}$ on $C$ such that
\begin{equation}\label{aij=ext1}
\dim \Ext ^1(E_j,E_i)=a_{ij}.
\end{equation}
By \eqref{RR formula} the equation \eqref{aij=ext1} is equivalent to the equation
\begin{equation}\label{nec cond}
a_{ij}=r_jr_i(g-1)-r_jd_i+r_id_j+\delta _{ij}\quad (=-\chi (E_j,E_i)+\delta _{ij})
\end{equation}
where $r_i=r(E_i)$ and $d_i=d(E_i)$. This gives us a necessary condition for the quiver $Q(A)$ to be realized on a curve of genus $g$. Actually this condition is also sufficient.
The following theorem answers question {\bf Q1} above.

\begin{theo} \label{answer question one} Let $g\geq 2$ and let $A=(a_{ij})\in M_{n\times n}(\bbZ)$ be a matrix with nonnegative entries. Then the quiver $Q(A)$ is realized by a torsion-free category on a given curve~$C$ of genus $g$ if and only if the following holds: there exists a collection of integers $(r_1,\ldots ,r_n,d_1,\ldots ,d_n)\in \bbZ _{>0}^n\times \bbZ ^n,$
such that
$$a_{ij}=r_jr_i(g-1)-r_jd_i+r_id_j+\delta _{ij}$$
for each pair $(i,j)$.
\end{theo}

\begin{proof}
We already explained the  ``only if'' direction.
For the ``if'' direction we will prove the following: given a set of integers $(r_1,\ldots ,r_n,d_1,\ldots ,d_n)\in \bbZ _{>0}^n\times \bbZ ^n$ such that for each pair $(i,j)$
\begin{equation}
\label{inequal}
r_jr_i(g-1)-r_jd_i+r_id_j+\delta _{ij}\geq 0
\end{equation}
holds,
on any curve $C$ of genus $g$ there exists a vertex-like collection of vector bundles
$\{E_1,\ldots ,E_n\}$ with $d(E_i)=~d_i$ and $r(E_i)=r_i$.
Choose vector bundles $E_i$ with $d(E_i)=~d_i$ and $r(E_i)=r_i$.
For $i\neq j$ the condition \eqref{inequal} means that $\chi (E_j,E_i)\leq 0$. Recall a theorem by Hirschowitz (see~\cite[Th. 1.2]{RTB}): if $E,F$ are generic vector bundles of given rank and degree on a curve $C$ of genus $\geq 2$ and $\chi(E,F)\leq 0$ then $\Hom(E,F)=0$. Also a generic vector bundle $E$ (of given rank and degree) on $C$ is stable (see~\cite[Prop. 2.6]{NR}) and thus $\End(E)=\k$. It follows that for a generic choice of vector bundles $E_i$ with degree $d_i$ and rank $r_i$ we have $\dim \Hom (E_j,E_i)=\delta _{ij}$ and the collection $\{E_1,\ldots,E_n\}$ is vertex-like.
\end{proof}

Recall that the \emph{slope} of a vector bundle $E$ is $d(E)/r(E)$.
%We can reformulate Theorem \ref{answer question one} in the following way.

\begin{prop} Let $C$ be a curve of genus $g\geq 2$ and let
$(r_1,\ldots ,r_n,d_1,\ldots ,d_n)\in \bbZ _{>0}^n\times \bbZ ^n$. Then there exists a vertex-like collection of vector bundles $\{E_1,\ldots ,E_n\}$ on $C$ with $d(E_i)=d_i$ and $r(E_i)=r_i$ if and only if for all $1\leq i,j\leq n$ we have
\begin{equation}
\label{reformulation}
\left|\frac{d_i}{r_i}-\frac{d_j}{r_j}\right|\le g-1.
\end{equation}
\end{prop}
\begin{proof} Indeed the inequality \eqref{inequal} for pairs $(i,j)$ and $(j,i)$  with $i\neq j$ is equivalent to the condition \eqref{reformulation}. Now the proposition follows by the same argument as in the proof of Theorem \ref{answer question one}.
\end{proof}

\begin{remark} If a quiver $Q(A)$ for a matrix $A=(a_{ij})$
is realizable by a torsion-free category on a curve of genus $g\geq 2$, then for all pairs of  indices $\{i,j\}$ at least one of the numbers $a_{ij}, a_{ji}$ is positive. In particular, the quiver $Q(A)$ is connected (compare with Proposition \ref{g=0 and g=1} for $g=1$ case). Indeed,
this follows from Theorem \ref{answer question one}.
\end{remark}

\subsection{Some uniqueness and non-uniqueness results}

\begin{prop}
\label{prop_unique}
Let $C,C'$ be curves, $g(C),g(C')\geq 2$. Let $E_1,\ldots,E_n$ and $E'_1,\ldots,E'_{n'}$ be two vertex-like families of vector bundles on $C$ and $C'$ respectively.  Assume that
$$\Phi\colon \langle E_1,\ldots,E_n\rangle\to \langle E'_1,\ldots,E'_{n'}\rangle$$
is an equivalence between the corresponding quiver-like categories. Then $n=n'$
 and there exist a permutation $\s\in S_n$ and $m\in\bbZ$ such that $\Phi(E_i)\simeq E'_{\s(i)}[m]$ for all $i=1,\ldots,n$.
\end{prop}
\begin{proof}
By Proposition \ref{useful summary for torsion-free},  the Grothendieck group  $K_0(\langle E_1,\ldots,E_n\rangle)$ is freely generated by the classes $[E_1],\ldots,[E_n]$ and similarly for  $K_0(\langle E'_1,\ldots,E'_{n'}\rangle)$. The equivalence $\Phi$ induces an isomorphism
$K_0(\langle E_1,\ldots,E_{n}\rangle)\simeq K_0(\langle E'_1,\ldots,E'_{n'}\rangle)$, which implies that $n=n'$.

For any $i$ the object $\Phi (E_i)$ is indecomposable and so by Proposition \ref{useful summary for torsion-free} we have
\begin{equation}
\label{eq_EFm}
\Phi(E_i)\simeq F_i[m_i]
\end{equation}
for some vector bundle $F_i$ on $C'$ and $m_i\in\bbZ$.

The equation  \eqref{RR formula} implies that
$$r(E_i)^2(1-g(C))=\chi(C,E_i,E_i)=\chi(C',F_i[m_i],F_i[m_i])=\chi(C',F_i,F_i)=r(F_i)^2(1-g(C')).$$
It follows that the ratio $r(F_i)/r(E_i)=:r/r'$ does not depend on $i$, where we denote
$$r=\sqrt{g(C)-1},\ \  r'=\sqrt{g(C')-1}$$
(recall that $g(C),g(C')\geq 2$).
By Proposition~\ref{useful summary for torsion-free}, we have
$$[F_i]=\sum_jb_{ij} [E'_j]\in  K_0(\langle E'_1,\ldots,E'_{n'}\rangle)$$
for some $b_{ij}\in\bbZ_{\ge 0}$. Moreover, the matrix $B=(b_{ij})$ is invertible over $\bbZ$.
It follows that
$$r(F_i)=\sum_j b_{ij} r(E'_j),\quad r\cdot r(E_i)=\sum_j b_{ij} r'\cdot r(E'_j).$$
%Similarly we see that
%$$r'\cdot r(E'_i)=\sim_j b_{ij} r\cdot r(E_j)$$
%for some invertible matrix $(b_{ij})\in M_n(\bbZ)$ with $b_{ij}\ge 0$.

Denote $s_i:=r\cdot r(E_i)$ and $s'_i:=r'\cdot r(E'_i)$. We have now
$$\sum_i s_i=\sum_{ij}b_{ij}s'_j=\sum_j(s'_j\sum_i b_{ij})\ge \sum_j s'_j,$$
because $\sum_i b_{ij}\ge 1$ (since $B$ is non-degenerate and $b_{ij}\in\bbZ_{\ge 0}$ for all $i,j$). Similarly we have $\sum_j s'_j\ge \sum_i s_i$. It follows that
$\sum_i b_{ij}=1$ for any $j$ and thus $B$ is a permutation matrix.

Hence $[F_i]=[E'_{\s(i)}]$ in $K_0(\langle E'_1,\ldots,E'_n\rangle)$ for some $\s\in S_n$
and all $i$. It follows from Proposition \ref{useful summary for torsion-free} that $\Phi(E_i)\simeq E'_{\s(i)}[m_i]$.
Now Lemma~\ref{no shifts} implies that all $m_i$ in \eqref{eq_EFm} are equal.
\end{proof}

\begin{lemma}\label{no shifts} Let $C$ be a curve of genus $g\geq 2$. Let $E_1,\ldots ,E_n$ be vector bundles on~$C$ such that for some integers $m_i$ the objects $\{E_1[m_1],\ldots ,E_n[m_n]\}$ form a vertex-like collection. Then $m_i=m_j$ for all $i,j$.
\end{lemma}

\begin{proof} It suffices to prove that $m_1=m_2$.
Formula \eqref{RR formula} implies that
$$\chi (E_1,E_2)+ \chi (E_2,E_1)=2(1-g)r(E_1)r(E_2)<0.$$
It follows that at least one of $\chi (E_1,E_2),\chi (E_2,E_1)$ is negative. Assume
$\chi (E_1,E_2)<0$, then
\begin{equation}
\label{condition1}
\Ext^1(E_1,E_2)\ne 0.
\end{equation}

By our assumption
\begin{equation}
\label{condition2}\Ext ^i(E_1[m_1],E_2[m_2])\neq 0\quad \text{implies that $i=1$}.
\end{equation}
Equations \eqref{condition1} and \eqref{condition2} imply that $m_1=m_2$.
\end{proof}

\begin{remark} Lemma \ref{no shifts} also holds for $g=0$, but it fails for $g=1$: if $\cL _1,\cL _2$ are distinct line bundles of the same degree, then for any $m$ the objects $\cL _1$ and $\cL _2[m]$ are orthogonal and hence the collection $\{\cL _1,\cL _2[m]\}$ is vertex-like.
\end{remark}

The following is an answer to question {\bf Q2} above.

\begin{cor}
\label{cor_unique}
Let $Q,Q'$ be quivers. Assume that   $Q$ is realizable and there is an equivalence
$D^b_0(Q)\simeq D^b_0(Q')$ (hence $Q'$ is also realizable). Then $Q\simeq Q'$.
\end{cor}
\begin{proof} Assume that the quiver $Q$ is realized on a curve of genus $g$. We consider several cases.

Case 1: $g=0$ and $Q$ is realized by a torsion-free category. Then by Proposition \ref{g=0 and g=1} we know that
$$D^b_0(Q)\simeq D^b_0({\bf Q}_0)\simeq D^b(\modd \k).$$
It follows that $Q={\bf Q}_0=Q'$.

Case 2: $g=1$ or $Q$ is realized by a torsion category. Then by Proposition \ref{g=0 and g=1} and Lemma \ref{first example for curves} there exists an equivalence of categories
$$\Psi \colon D^b_0(Q)\stackrel{\sim}{\to} D^b_0({\bf Q}_1\sqcup\ldots \sqcup {\bf Q}_1).$$
Comparing the $K$-groups of these categories we find that the two quivers have the same number of vertices, say $n$. Let $s_1,\ldots ,s_n\in D^b_0(Q)$ (resp. $s'_1,\ldots ,s'_n\in D^b_0({\bf Q}_1\sqcup\ldots \sqcup {\bf Q}_1)$) be the collection of simple modules corresponding to vertices. Note that the objects $s'_i\in D^b_0({\bf Q}_1\sqcup\ldots \sqcup {\bf Q}_1)$ are characterized (up to a shift) by the property that $\End (s'_i)=\k$. It follows that there exists a permutation $\s \in S_n$ such that for each $i$
$$\Psi (s_i)=s'_{\s (i)}[m_i]\quad \text{for some $m _i \in \bbZ$}.$$
Moreover, $D^b_0({\bf Q}_1\sqcup\ldots \sqcup {\bf Q}_1)$ is the orthogonal sum of its subcategories $\langle s_i'\rangle$. It follows that $D^b_0(Q)$ is also the orthogonal sum of its subcategories $\langle s_i\rangle$. Therefore $Q\simeq {\bf Q}_1\sqcup\ldots \sqcup {\bf Q}_1$ and similarly $Q'\simeq {\bf Q}_1\sqcup\ldots \sqcup {\bf Q}_1$.

Case 3: $g\ge 2$ and $Q$ is realized by a torsion-free category.
In this case an isomorphism $Q\simeq Q'$ follows from Proposition \ref{prop_unique}.
\end{proof}

\begin{remark}
Note that Corollary~\ref{cor_unique} does not hold in general if $Q$ is not assumed to be realizable. For example, let $Q$ be any tree and let $Q'$ be the quiver obtained from~$Q$ be reversing some arrows. Then $D^b_0(Q)=D^b(\modd \k Q)$, $D^b_0(Q')=D^b(\modd \k Q')$ and it is well-known that the categories $D^b(\modd \k Q)$ and $D^b(\modd \k Q')$  are equivalent.
\end{remark}

\end{document}